\documentclass{article}
\usepackage[utf8]{inputenc}
\usepackage{graphicx,graphics}

\usepackage{rotating}
\usepackage[twoside]{geometry}
\usepackage{epsfig}
\usepackage{cmap}
\usepackage{appendix}
\usepackage{pgfplots} 
\pgfplotsset{compat=1.18}

\usepackage{graphicx,graphics}

\usepackage{hyperref}
\hypersetup{
pdftitle={Aloisiopaper},            
colorlinks=true,      
linkcolor=black,         
citecolor=blue}

\usepackage{amssymb,amsfonts,amstext,amsthm}
\usepackage{mathrsfs}
\usepackage[intlimits]{amsmath}
\newtheorem{proposition}{\bf Proposition}[section]

\newtheorem{lemma}{\bf Lemma}[section]
\newtheorem{theorem}{\bf Theorem}[section]

\newtheorem{remark}{Remark}[section]

\newcounter{SE}

\newcommand{\bb}[1]{\mathbb{#1}}

\title{Dynamical localization and eigenvalue asymptotics: long-range hopping lattice operators with electric field}
\author{M. Aloisio\thanks{Department of Mathematics and Statistics, UFVJM, Diamantina, MG, Brazil} \thanks{Institute of Mathematics and Computer Sciences (ICMC), USP, São Carlos, SP, Brazil\\ Corresponding author. Email: ec.moacir@gmail.com}}

\begin{document}

\maketitle

\begin{abstract} We prove power-law dynamical localization for polynomial long-range hopping lattice operators with uniform electric field under any bounded perturbation. Actually, we introduce new arguments in the study of dynamical localization for long-range models with unbounded potentials, involving the Min-Max Principle and a notion of power-law ULE. Unlike existing results in the literature, our approach does not rely on KAM techniques or on Green’s function estimates, but rather on the asymptotic behavior of the eigenvalues and the potential. It is worth underlining that our general results can be applied to other models, such as Maryland-type potentials.
\end{abstract}  

\ 

\noindent{\bf Keywords}:  dynamical localization, discrete spectrum, spectral theory, quantum dynamics.

\ 

\noindent{\bf  AMS classification codes}: 28A80 (primary), 42A85 (secondary).    

\renewcommand{\thetable}{\Alph{table}}


\section{Introduction}

\subsection{Contextualization}\label{sectIntrod}
\

There is a vast literature concerning the dynamics of wave packet solutions \( e^{-itH}\psi \) of the Schr\"odinger equation
\begin{equation}\label{SE}
    \begin{cases} \partial_t \psi = -iH\psi, ~t \in \mathbb{R}, \\
    \psi(0) = \psi, ~\psi \in \text{dom}H, \end{cases}
\end{equation}
where \( H \) is a self-adjoint operator in $\ell^2(\mathbb{Z})$. Actually, the study of the dynamics \( e^{-itH}\psi\) through the spectral properties of $H$ is a classical subject in quantum dynamics; for instance, see \cite{Aizenman,Oliveira,DJLS,Germinet,Last,AbstSimon,Stolz,Tcheremchantsev} and references therein. In this context, a dynamical quantity to probe the asymptotic behavior of \( e^{-itH}\psi \) is the $q$-moment, $q > 0$, of the position operator at time \( t > 0 \), defined as
\begin{equation*}
    \sum_{n \in \mathbb{Z}} |n|^q |\langle e^{-itH}\psi ,\delta_n \rangle|^2,     
\end{equation*}
where \( \delta_n = (\delta_{nj})_{j \in \mathbb{Z}} \)  is the canonical basis of \( \ell^2(\mathbb{Z})\). These quantities characterize the spreading of the wave packet \( e^{-itH}\psi \).

In recent years, there has been substantial research activity particularly dedicated to the study of the phenomenon of \textit{dynamical localization} \cite{Avila,Damanik,DamanikFillman1,DamanikFillman2,Pigossi3,Pigossi2,Pigossi,Sun2,shi1,shi2,shi3,shi4,Sun}, measured through the moments of order \( q > 0 \).  We say that \( H \) exhibits dynamical localization if, for all initial conditions \( \delta_k \), each moment is uniformly bounded in time, that is, for each \( q > 0 \) and each \(\delta_k\),
\begin{equation*}\label{B}
   \sup_{t \in \mathbb{R}} \sum_{n \in \mathbb{Z}} |n|^q |\langle e^{-itH}\delta_k ,\delta_n \rangle|^2 < +\infty.    
\end{equation*}

This is the standard strong notion of absence of transport for discrete operators. Due to the RAGE Theorem \cite{Oliveira}, dynamical localization requires pure point spectrum; however, it is possible to have systems with pure point spectrum without dynamical localization \cite{DJLS,Germinet}. One way to establish such a property is through the control of eigenfunctions, referred to as semi-uniform localization of eigenfunctions (SULE) \cite{DJLS, Tcheremchantsev}.

Let $H$ be a self-adjoint operator acting in \( \ell^2(\mathbb{Z}) \). We say that \( H \) exhibits SULE if there exists an orthonormal basis of \( \ell^2(\mathbb{Z}) \) by eigenfunctions of \( H \), \( \{\phi_m\}_{m \in \mathbb{Z}} \), that is, \( H\phi_m = \lambda_m \phi_m, \, m\in \mathbb{Z} \), for which there exist \( \alpha > 0 \) and, for each eigenfunction \( \phi_m \), an \( j_m \in \mathbb{Z} \) such that, for any \( \delta > 0 \),
\[
|(\phi_m)(n)| \leq \gamma_\delta e^{\delta |j_m|-\alpha|n-j_m|},
\]
where \( \gamma_\delta >0\) is uniform in \( m \) and \( n \). Consequently, if \( H \) satisfies SULE, then $H$ has dynamical localization (Theorem 7.5 in \cite{DJLS}). Choosing $\delta = 0$ and $j_m = m$ in the definition above yields the 
notion of ULE (uniform localization of eigenfunctions), namely,
\[
|\phi_m(n)| \le \gamma e^{-\alpha |n - m|}, \qquad m,n \in \mathbb{Z}.
\]
It is worth underlining that, in general, ergodic  models cannot satisfy ULE due to the presence of resonances \cite{DJLS}.

An important model in quantum dynamics that satisfies ULE is the discrete one-dimensional Schrödinger operator with uniform electric field of strength \( E = 1 \) \cite{Pigossi, Nazareno}:
\[
\text{dom } H_0 := \{u\in \ell^2(\mathbb{Z}): \sum_{n\in \mathbb{Z}}|n|^2|u(n)|^2 < +\infty\},
\]
\begin{equation}\label{operadorcomcal}
(H_0u)(n) =  u(n-1) + u(n+1) + n u(n), \quad  n \in \mathbb{Z}.    
\end{equation}

de Oliveira and Pigossi \cite{Pigossi} proved that \( H_0 \) exhibits uniform decay of eigenfunctions. For this model, eigenvalues and eigenfunctions were explicitly computed using Fourier series in \cite{Pigossi}, revealing that the spectrum consists of the set of integers. In particular, it was established that \( H_0 \) exhibits ULE (uniform localization of eigenfunctions) and, consequently, dynamical localization. Moreover, through an iterative matrix diagonalization process originally developed by Kolmogorov, Arnold, and Moser (KAM), it was shown that for sufficiently small perturbations, the uniform decay of eigenfunctions persists, thereby establishing the persistence of dynamical localization. In this case, it was also shown that the asymptotic behavior of the eigenvalues is preserved. In fact, the following result was established.

\begin{theorem}[Theorem 4.2 in \cite{Pigossi}]\label{pigo0thm}  To every $b \in \ell^\infty(\mathbb{Z},\mathbb{R})$, the operator  $H_0 + b$  has purely discrete spectrum; if $||b||_{\infty}$ is small enough, then there exists an orthonormal basis of \( \ell^2(\mathbb{Z}) \) by eigenfunctions of \( H_0 + b, \) \( \{\phi_m\}_{m \in \mathbb{Z}}, \) that satisfies
\[
|\phi_m(n)| \leq \gamma \, e^{-|m-n|}, \quad n,m \in \bb{Z},    
\]
for some $\gamma>0$; in particular, $H_0 +b$ has dynamical localization. Moreover, there exists a $\gamma'>0$ such that its eigenvalues satisfy
\[
|\lambda_n - n| \leq \gamma',  \quad n \in \mathbb{Z}.    
\]
\end{theorem} 

\begin{remark}{\rm Actually, de Oliveira and Pigossi in \cite{Pigossi2}, using an iterative matrix diagonalization process (KAM), proved the persistence of the uniform decay of eigenfunctions for sufficiently small perturbations of self-adjoint operators with discrete spectrum that possess this property, since the modulus of the difference between the eigenvalues of the original operator is lower bounded by a uniform constant. For example, in the case of \( H_0 \), one has \( |\lambda_m - \lambda_n| = |m - n| \geq 1 \) for \( m \neq n \).}
\end{remark}

Still with respect to models with uniform electric field, consider the polynomial long-range hopping lattice operators.

Let $(a_n)_{n \in \mathbb{Z}}$ in ${\mathbb{C}}^{\mathbb{Z}}$ such that $a(0) = 0$, $a(m) = a^*(-m)$, $m \in \mathbb{Z}$, and  \(a \in \ell^1_r(\mathbb{Z})\), $r\geq 0$, that is, 
\[||a||_r = \sum_{m \in \mathbb{Z}} |a(m)||m|^r  < + \infty.\]
Let ${\mathcal{L}}_0$ given by
\[
\text{dom } {\mathcal{L}}_0 := \{u\in \ell^2(\mathbb{Z}): \sum_{n\in \mathbb{Z}}|n|^2|u(n)|^2 < +\infty\},
\]

\begin{equation*}\label{eqlongalc}
({\mathcal{L}}_0u)(n) = (T_a u)(n)+ n u(n) = \sum_{m \in \mathbb{Z}} a(n-m) u(m) + n u(n), \quad  n \in \mathbb{Z}.    
\end{equation*}

Recently, Sun and Wang in \cite{Sun} established the following results regarding localization for this model with uniform electric field.

\begin{theorem}[Theorem 1.1 in \cite{Sun}]\label{sunthm0} Fix \(r > 1\) and let \(a \in \ell^1_r(\mathbb{Z})\).  To every $b \in \ell^\infty(\mathbb{Z},\mathbb{R})$, ${\mathcal{L}}_0 + b$ has purely discrete spectrum.  If $\displaystyle\frac{1}{2} < s < r - \displaystyle\frac{1}{2}$ and  $||b||_{\infty}$ is small enough, then there exists an orthonormal basis of \( \ell^2(\mathbb{Z}) \) by  eigenfunctions of \( {\mathcal{L}}_0 + b, \) \( \{\phi_m\}_{m \in \mathbb{Z}},\) that satisfies
\[
|\phi_m(n)| \leq  \frac{\gamma_{s,r}}{|n-m|^s},
\]
where \(\gamma_{s,r}\) is a constant that depends only on \(s\) and \(||a||_r\).
\end{theorem}

\begin{theorem}[Theorem 1.2 in \cite{Sun}]\label{sunthm}  Fix \(r > 1\).  Let ${\mathcal{L}}_0 + b$ be as before. If
\begin{equation}\label{eq004}
0 \leq q < s - \frac{1}{2} < r - 1   
\end{equation}
and  $||b||_{\infty}$ is small enough, then for any \(\psi \in \ell^2_q(\mathbb{Z})\), one has
\[
\sup_{t \in \mathbb{R}} \sum_{n \in \mathbb{Z}} |n|^{2q} |\langle e^{-i({\mathcal{L}}_0+b)t} \psi,\delta_n \rangle|^2  < +\infty. 
\]  
Moreover, there exists a $\gamma'>0$ such that the eigenvalues satisfy
\[
|\lambda_n - n| \leq \gamma', \quad n \in \mathbb{Z}.    
\]
\end{theorem}

\begin{remark}
\end{remark}

\begin{enumerate}
\item [i)] If \(q > \frac{1}{2}\), the assumption \eqref{eq004} in Theorem 1.2 can be replaced by \(\frac{1}{2} < q \leq s < r - \frac{1}{2}\).

\item [ii)] Note that \( H_0 \) is a particular case of a long-range hopping lattice operator. Thus, the results of Sun and Wang presented above extend the result of Pigossi and de Oliveira. However, it is worth mentioning that, although Sun and Wang obtained uniform power-law dynamical localization, this differs from the result of Pigossi and Oliveira, where exponential localization directly leads to dynamical localization without restrictions on \( q > 0 \).

\item [iii)] Pigossi and de Oliveira in \cite{Pigossi3} have also recently proven exponential dynamical localization for polynomial long-range hopping lattice operators when \( a(m) \) has finite support.    
\end{enumerate}

Unlike ergodic Schr\"odinger operators (such as the Anderson model), where small divisors or resonances may occur (see \cite{Avila,Damanik,DamanikFillman1,DamanikFillman2,Spencer,shi1,shi2} and references therein), the model considered here lacks resonances, making it possible to establish rigidity results that relate the spectral structure to the decay of eigenfunctions.

We shall prove the following

\begin{theorem}\label{mainresult} Let ${\mathcal{L}}_0$ be as before and let ${\mathcal{L}}_0+b$, $b \in \ell^\infty(\mathbb{Z},\mathbb{R})$. Then,

\begin{enumerate}

\item [i)] there exists a $\gamma>0$ such that the eigenvalues satisfy
\begin{equation*}
|\lambda_n - n| \leq \gamma, \quad n \in \mathbb{Z},    
\end{equation*}

\item [ii)] if $a \in \ell^1_r(\mathbb{Z})$, $r \in \mathbb{N} \cup \{0\}$,  then  any  orthonormal basis of \( \ell^2(\mathbb{Z}) \) by  eigenfunctions of \( {\mathcal{L}}_0 + b, \) \( \{\phi_m\}_{m 
\in \mathbb{Z}},\) satisfies
\begin{equation*}
|\phi_m(n)| \leq  \frac{\gamma_{r}}{|n-m|^{r+1}},
\end{equation*}
where \(\gamma_{r}\) is a constant that depends only on r and \(||a||_r\),

\item[iii)] if $0< q < 2r-1$, then for any \(\psi = \delta_k\), one has
\[
\sup_{t \in \mathbb{R}} \sum_{n \in \mathbb{Z}} |n|^{q} |\langle e^{-i({\mathcal{L}}_0+b)t} \psi,\delta_n \rangle|^2  < +\infty. 
\]  

\end{enumerate}
\end{theorem}

\begin{remark}
\end{remark}

\begin{enumerate}

\item[i)] Note that if $T_a = \triangle$ is the free Laplacian, then $a \in \ell^1_r(\mathbb{Z})$ for every $r > 0$. Thus, our result implies that, for any $b \in \ell^\infty(\mathbb{Z}, \mathbb{R})$, the operator $H_0 + b$ (see definition in~\eqref{operadorcomcal}) exhibits dynamical localization, that is, all moments of order $q>0$ are uniformly bounded. This answers a question posed by de Oliveira and Pigossi in~\cite{Pigossi}.

\item[ii)] Although the operator ${\mathcal{L}}_0 + b$, for any bounded potential $b$, has purely discrete spectrum, which provides strong evidence of dynamical localization, it remained an open question whether dynamical localization still holds when $||b||_{\infty}$ is large. Shortly after the announcement of the first version of this manuscript,  Sun and Wang in \cite{Sun3} announced an affirmative answer to this question in the case of Jacobi operators, using estimates on the Green’s function and without resorting to KAM techniques. Our approach also does not rely on KAM techniques. However, we do not use Green’s function estimates. Instead, our proof is based solely on the asymptotic behavior of the eigenvalues and the potential. Actually, in this paper, we introduce new arguments in the study of dynamical localization for long-range models with unbounded potentials, involving the Min-Max Principle (Theorem \ref{minmaxthm}) and a notion of power-law ULE (Lemma \ref{PLSULEthm}).

\item[iii)] Note that our result applies to the case where $T_a$ is the fractional Laplacian \cite{Coni}.

\item[iv)] Our main contribution in this paper is to show that there is a symbiotic relation between the structure of the spectrum/potential and the uniform polynomial decay of eigenfunctions for models with an electric field, which directly leads (Lemma \ref{PLSULEthm}) to power-law dynamical localization (Theorem \ref{mainresult}). To the best of the present author's knowledge, none of this has yet been detailed in the literature. It is worth underlining that our strategy for proving Theorem \ref{mainresult} is different from the proofs of Theorems \ref{sunthm0} and \ref{sunthm} in \cite{Sun}. In fact, here we use a theory on the asymptotic behavior of eigenvalues of the operators (Min-Max Principle), combined with a notion of power-law ULE (Lemma \ref{PLSULEthm}), whereas in \cite{Sun} the authors deal with Sobolev-type norms and KAM techniques. It is worth mentioning that our notion of power-law ULE (Lemma \ref{PLSULEthm}) applies to other models, such as Maryland-type potentials \cite{SimonMary} (see Lemma \ref{PLSULEthm} below and Corollary 2.10 in \cite{shi2}).

\item[v)] It would be interesting to analyze the technique presented in this work in the context of time-periodic perturbations of models with a uniform electric field, which were discussed in \cite{Pigossi1,Hu}.

\end{enumerate}


\subsection{Organization of text}
\ 

In Section \ref{subasynto}, we discuss the stability of the asymptotic behavior of eigenvalues of models with an electric field. Section \ref{subSULE} addresses how a variation of ULE implies a variation of dynamical localization (Lemma \ref{PLSULEthm}), and then we prove Theorem  \ref{mainresult}. Appendix \ref{secappendix} is devoted to the proof of Proposition \ref{discretproposition}.

Some words about notation: \( \ell^2(\mathbb{Z}) \) represents the space of square-summable sequences indexed by the integers \( \mathbb{Z} \), that is,  
\[
\ell^2(\mathbb{Z}) = \left\{ \psi: \mathbb{Z} \longrightarrow \mathbb{C}\; : \; \sum_{n \in \mathbb{Z}} |\psi(n)|^2 < +\infty \right\}.
\]  
This space is equipped with the usual inner product given by  
\[
\langle \psi, \phi \rangle = \sum_{n \in \mathbb{Z}} \overline{\psi(n)} \phi(n), \quad  \psi, \phi \in \ell^2(\mathbb{Z}).
\]  
The induced norm is defined as \( \|\psi\| = \sqrt{\langle \psi, \psi \rangle} \). 

In this paper, $I$  will always denote the identity operator and \( \|f\|_{\infty} \) the norm of a sequence \( f = (f(n))_{n \in \mathbb{Z}} \) in 
\[
\ell^\infty(\mathbb{Z},\mathbb{R}) = \left\{ \psi: \mathbb{Z} \longrightarrow \mathbb{R}\; : \; \sup_{n \in \mathbb{Z}} |\psi(n)| < +\infty \right\},
\]  
which is given by 
\[\|f\|_{\infty} = \sup_{n \in \mathbb{Z}} |f(n)|.\]

Here, $\{\delta_n := (\delta_{nj})_{j \in \mathbb{Z}} \}$ will always denote the canonical basis in $\ell^2(\mathbb{Z})$. Let \( H \) be a self-adjoint operator defined in \( \ell^2(\mathbb{Z}) \). The resolvent set of \( H \), denoted by \( \rho(H) \), is the set of values \( \lambda \in \mathbb{C} \) for which \( (H - \lambda I)^{-1} \) exists and is a bounded operator. The spectrum of \( H \) is defined as  
\[
\sigma(H)= \mathbb{C} \setminus  \rho(H).
\]  

We introduce the following standard decomposition of the spectrum:

\begin{enumerate}
\item[i)] The essential spectrum of $H$ is the set $\sigma_{\mathrm{ess}}(H)$ consisting of the
accumulation points of $\sigma(H)$ together with the eigenvalues of $H$ of infinite multiplicity.

\item[ii)] The discrete spectrum of $H$ is the set
\[
\sigma_{\mathrm{disc}}(H) := \sigma(H) \setminus \sigma_{\mathrm{ess}}(H),
\]
that is, the set of isolated eigenvalues of $H$, each of finite multiplicity.

\item[iii)] If $\sigma_{\mathrm{ess}}(H) = \varnothing$, we say that $H$ has purely discrete spectrum;
if $\sigma_{\mathrm{disc}}(H) = \varnothing$, we say that $H$ has purely essential spectrum.
\end{enumerate}


\section{Eigenvalue asymptotics}\label{subasynto}

\ 

In this section, we show, as a consequence of the Min-Max Principle (Theorem \ref{minmaxthm}), that the eigenvalues of a bounded self-adjoint operator perturbed by the linear potential \( V(n) = n \) exhibit the same asymptotic behavior as the potential \( V \). For this, some preparation is required.

Let $A$ be a bounded self-adjoint operator in $\ell^2(\mathbb{Z})$ and let  $V : \mathbb{Z} \to \mathbb{R}$ be any real-valued function. Consider 
\[
\text{dom } H := \{u\in \ell^2(\mathbb{Z}): \sum_{n\in \mathbb{Z}}|V(n)|^2|u(n)|^2 < +\infty\},
\]
\[
(Hu)(n) =  (Au)(n) + V(n)u(n), \quad  n \in \mathbb{Z};
\]
so, by the Kato-Rellich Theorem \cite{Oliveira}, $H$ is self-adjoint.

Although natural for specialists, we initially introduce the result below, which is a consequence of the second resolvent identity \cite{Oliveira} and the fact that operators with compact resolvent have purely discrete and unbounded spectrum (see Theorem \ref{thmdiscret} in Appendix \ref{secappendix}).

\begin{proposition} \label{discretproposition} Let $A$ be a bounded self-adjoint operator in $\ell^2(\mathbb{Z})$ and let $V: \mathbb{Z} \longrightarrow \mathbb{R}$ with 
\[ \displaystyle\lim_{n \to \mp \infty} V(n) = \mp\infty;\]
 so $H= A+V$ has purely discrete spectrum and its eigenvalues satisfy
\[\displaystyle\lim_{n \to \mp\infty} \lambda_n = \mp\infty.\]
In this case, one writes
\begin{equation*}
\dots,\lambda_{-(n+1)} \leq \lambda_{-n}, \dots, \lambda_{-2} \leq \lambda_{-1} \leq \lambda_0 \leq \lambda_1 \leq \lambda_2, \dots, \lambda_n \leq \lambda_{n+1},  \dots 
\end{equation*}
where $\lambda_0 = \displaystyle\min \{\lambda_n \mid \,  \lambda_n \geq 0 \}$.
\end{proposition}

\begin{remark}\label{remrkeig}
\end{remark}
\begin{enumerate}

    \item [i)] Condition $\displaystyle\lim_{n \to \pm \infty} |V(n)| = +\infty$
is important to ensure that \( H \) has a purely discrete spectrum. As is well known, e.g., only $\displaystyle\limsup_{n \to \pm \infty} |V(n)| = +\infty$ is not sufficient; see, for example, \cite{Jitomirskaya} for cases involving  Schrödinger operators with sparse potentials.

    \item [ii)] For the reader's convenience, we present a proof of this result in Appendix \ref{secappendix}. 
\end{enumerate}

The result below is the Theorem XIII.1 in \cite{Simon2}.

\begin{theorem}[Min-Max Principle]\label{minmaxthm} Let $H$ be a self-adjoint operator that is bounded from below, i.e.,
$H \ge cI$ for some $c$. Define
\[
\mu_n
=
\sup_{\varphi_1,\ldots,\varphi_{n-1}}
\inf_{\substack{
\varphi \perp \varphi_1,\dots,\varphi_{n-1} \\
\varphi \in {\rm dom}\,H,\ \|\varphi\|=1
}}
\langle H \varphi,\varphi\rangle.
\]
Then, for each fixed $n$, either:
\begin{enumerate}
\item[(a)]
there are $n$ eigenvalues (counting degenerate eigenvalues a number of
times equal to their multiplicity) below the bottom of the essential
spectrum, and $\mu_n$ is the $n$th eigenvalue counting multiplicity;
\item[(b)]
$\mu_n$ is the bottom of the essential spectrum, i.e.,
\[
\mu_n = \inf\{\lambda \mid \lambda \in \sigma_{\mathrm{ess}}(H)\}
\]
and in that case
\[
\mu_n = \mu_{n+1} = \mu_{n+2} = \cdots
\]
and there are at most $n-1$ eigenvalues (counting multiplicity) below
$\mu_n$.
\end{enumerate}
\end{theorem}

\begin{theorem}\label{mainthm}
Let $H = A + V$ on $\ell^2(\mathbb{Z})$, where $A$ is bounded  self-adjoint  and
\[
(Vu)(n) = n\,u(n), \qquad n \in \mathbb{Z}.
\]
Then $H$ is self-adjoint with purely discrete spectrum $\{\lambda_n\}_{n\in\mathbb{Z}}$. In this case, one writes
\begin{equation*}
\dots\lambda_{-n}, \dots, \lambda_{-2} \leq \lambda_{-1} \leq \lambda_0 \leq \lambda_1 \leq \lambda_2, \dots, \lambda_n  \dots 
\end{equation*}
where $\lambda_0 = \displaystyle\min \{\lambda_n \mid \,  \lambda_n \geq 0 \}$. Moreover,
\[
|\lambda_n - n| \le ||A||, \qquad n \in \mathbb{Z}.
\]
\end{theorem}

\begin{proof} The first two statements in the theorem follow from Proposition \ref{discretproposition}.

Since the Min–Max Principle applies only to operators which are bounded below, and $H$ itself is not bounded below, we will write $H = H_+  +  H_-$ as below.

Let $\{\phi_j\}_{j\in\mathbb{Z}}$ be an orthonormal basis of eigenvectors of $H$ with corresponding eigenvalues $\{\lambda_j\}_{j\in\mathbb{Z}}$. We decompose $H$ since
\[
H = H_+ + H_-,
\qquad
H_+ = \sum_{j>0} \lambda_{j} \langle \cdot,\phi_j\rangle\phi_j,
\quad
H_- = \sum_{j \leq 0 } \lambda_{j} \langle \cdot,\phi_j\rangle\phi_j.
\]
Then $H_+$ and $H_-$ are self-adjoint, acting on the orthogonal subspaces $\mathcal{H}_+ = \operatorname{ran}(H_+)$, and $\mathcal{H}_- = \operatorname{ran}(H_-)$, respectively. 

For each $N \geq 2$, we consider the following auxiliary operator
\[
H_+^N=\sum_{j>0}\lambda_j\langle\cdot,\phi_j\rangle\phi_j
+\sum_{j\le 0} \lambda_N \langle\cdot,\phi_j\rangle\phi_j.
\]

We also decompose $V$:
\[
V =  V_+ + V_-,
\quad
V_+ = \sum_{j > 0} j \langle \cdot,\delta_j\rangle\delta_j, \quad
V_- = \sum_{j \leq 0 } j \langle \cdot,\delta_j\rangle\delta_j.
\]
Then $V_+$ and $V_-$ are self-adjoint, acting on the orthogonal subspaces $\mathcal{V}_+ = \operatorname{ran}(V_+)$, and $\mathcal{V}_- = \operatorname{ran}(V_-)$, respectively.

For each $N \geq 2$, we also consider the following auxiliary operator
\[
V_+^N=\sum_{j > 0} j \langle \cdot,\delta_j\rangle\delta_j
+\sum_{j\le 0} N \langle\cdot,\delta_j\rangle\delta_j.
\]
\noindent {\bf Claim 1:} For every $\psi \in  \mathrm{dom}\, V_+ \cap \mathrm{dom}\,H_+^N$, $||\psi||=1$,
\[- ||A|| + \langle  V_+ \psi,\psi\rangle  \leq  \langle H_+^N\psi,\psi\rangle.\]

\noindent \textbf{Proof of Claim 1.} 
If $\psi \in \mathcal{V}_+^\perp$, then
\[
- ||A|| + \langle V_+ \psi, \psi \rangle = - ||A|| \leq 0 \leq  \langle H_+^N \psi, \psi \rangle.
\]
On the other hand, if $\psi \in \mathcal{V}_+$, since $A$ is bounded,
\[
- ||A|| + \langle V_+ \psi, \psi \rangle = - ||A|| + \langle V \psi, \psi \rangle  
\leq  \langle H \psi, \psi \rangle \leq  \langle H_+^N \psi, \psi \rangle.
\]
This proves Claim~1. 

Note that: since $\{\phi_j\}_{j \in \mathbb{Z}} \subset  {\rm dom}\,V \subset \mathcal{L} =  {\rm dom}\,V_+ \cap {\rm dom}\,H_+^N$; so, $\mathcal{L} $ is dense in ${\rm dom}\,H_+^N$ with respect to the graph norm $\|\cdot\|_{H_+^N}.$ Hence,
\[
\sup_{\psi_1,\dots,\psi_{n-1}} 
\inf_{\substack{ \psi \perp \psi_1,\dots,\psi_{n-1}\\ \psi\in \mathcal{L} \ \|\psi\|=1}} 
\langle  H_+^N\psi,\psi\rangle 
=
\sup_{\psi_1,\dots,\psi_{n-1}} 
\inf_{\substack{ \psi \perp \psi_1,\dots,\psi_{n-1}\\ \psi\in {\rm dom}\,H_+^N,\ \|\psi\|=1}} 
\langle  H_+^N\psi,\psi\rangle.
\]

By Claim~1,
\begin{eqnarray*}  
  -\|A\| + \sup_{\psi_1,\dots,\psi_{n-1} \subset {\mathcal{V}}_+}
\inf_{\substack{ \psi \perp \psi_1,\dots,\psi_{n-1}\\ \psi\in {\rm dom}\,V_+ \cap {\mathcal{V}}_+,\ \|\psi\|=1}} 
\langle  V_+\psi,\psi\rangle &=& -\|A\| +  \sup_{\psi_1,\dots,\psi_{n-1} \subset {\mathcal{V}}_+}
\inf_{\substack{ \psi \perp \psi_1,\dots,\psi_{n-1}\\ \psi\in {\rm dom}\,V_+,\ \|\psi\|=1}} 
\langle  V_+\psi,\psi\rangle\\  &\leq&  -\|A\| + \sup_{\psi_1,\dots,\psi_{n-1} \subset {\mathcal{V}}_+} \inf_{\substack{ \psi \perp \psi_1,\dots,\psi_{n-1}\\ \psi\in \mathcal{L},\ \|\psi\|=1}}  \langle V_+\psi,\psi\rangle\\ &\leq&  \sup_{\psi_1,\dots,\psi_{n-1} \subset {\mathcal{V}}_+}  \inf_{\substack {\psi \perp \psi_1,\dots,\psi_{n-1}\\ \psi\in \mathcal{L},\ \|\psi\|=1}} \langle H_+^N\psi,\psi\rangle\\ &\leq&  \sup_{\psi_1,\dots,\psi_{n-1}} \inf_{\substack {\psi \perp \psi_1,\dots,\psi_{n-1}\\ \psi\in \mathcal{L},\ \|\psi\|=1}} \langle H_+^N\psi,\psi\rangle.
\end{eqnarray*}
which gives, by the Min-Max Principle , for every $1 \leq n \leq  N-1$, 
\[-\|A\| + n \le \mu_n(H_+^N) = \lambda_n(H).\]
(Note that: $\sigma_{ess}(V_{+}\big|_{\mathcal{V}_{+}} ) = \emptyset$ and $\sigma_{ess}(H_+^N) = \{\lambda_N\}$!) Thus, since $N\geq 2$ is arbitrary, it follows that, for every $ n \geq 1$, 
\[-\|A\| + n \le \lambda_n(H).\]

\noindent {\bf Claim 2:} For every $\psi \in  \mathrm{dom}\, V_+^N \cap \mathrm{dom}\,H_+$, $||\psi||=1$,
\[\langle  H_+ \psi,\psi\rangle  \leq  ||A|| +  \langle V_+^N\psi,\psi\rangle.\]

\noindent \textbf{Proof of Claim 2.} 
If $\psi \in \mathcal{H}_+^\perp$, then
\[
\langle H_+ \psi, \psi \rangle = 0 \leq \|A\| + \langle V_+^N \psi, \psi \rangle.
\]
On the other hand, if $\psi \in \mathcal{H}_+$, since $A$ is bounded,
\[
\langle H_+ \psi, \psi \rangle = \langle H \psi, \psi \rangle 
= \langle A \psi, \psi \rangle + \langle V \psi, \psi \rangle 
\leq \|A\| + \langle V_+^N \psi, \psi \rangle.
\]
This proves Claim~2.

Note that: $\{\delta_j\}_{j \in \mathbb{Z}} \subset \mathrm{dom}\, V \subset \mathcal{D} =  \mathrm{dom}\, V_+^N \cap \mathrm{dom}\,H_+$. So, $\mathcal{D} $ is dense in ${\rm dom}\,V_+^N$ with respect to the graph norm $\|\cdot\|_{V_+^N}.$ Now, using the upper bound from  Claim 2, we obtain, for every $n \geq 1$,
\begin{eqnarray*}
\sup_{\psi_1,\dots,\psi_{n-1} \subset {\mathcal{H}}_+} \inf_{\substack{ \psi \perp \psi_1,\dots,\psi_{n-1}\\ \psi\in \mathrm{dom}\,H_+ \cap {\mathcal{H}}_+ ,\ \|\psi\|=1}} \langle  H_+\psi,\psi\rangle &=& \sup_{\psi_1,\dots,\psi_{n-1} \subset {\mathcal{H}}_+} \inf_{\substack{ \psi \perp \psi_1,\dots,\psi_{n-1}\\ \psi\in \mathrm{dom}\,H_+ ,\ \|\psi\|=1}} \langle  H_+\psi,\psi\rangle\\ &\leq& \sup_{\psi_1,\dots,\psi_{n-1} \subset {\mathcal{H}}_+} \inf_{\substack{ \psi \perp \psi_1,\dots,\psi_{n-1}\\ \psi\in \mathcal{D} ,\ \|\psi\|=1}} \langle  H_+\psi,\psi\rangle\\ &\leq&  ||A|| + \sup_{\psi_1,\dots,\psi_{n-1} \subset {\mathcal{H}}_+} \textit{}\inf_{\substack{ \psi \perp \psi_1,\dots,\psi_{n-1}\\ \psi\in \mathcal{D},\ \|\psi\|=1}} \langle  V_+^N \psi,\psi\rangle\\ &\leq&  ||A|| + \sup_{\psi_1,\dots,\psi_{n-1} } \inf_{\substack{ \psi \perp \psi_1,\dots,\psi_{n-1}\\ \psi\in \mathcal{D},\ \|\psi\|=1}} \langle  V_+^N \psi,\psi\rangle
\end{eqnarray*}
which gives, by the Min-Max Principle, for every $1 \leq n \leq N-1$, 
\[
\lambda_{n}(H) = \lambda_{n}(H_+) \le \|A\| + n.
\]
(Note that: $\sigma_{ess}(H_+\big|_{ {\mathcal{H}}_+} ) = \emptyset$ and $\sigma_{ess}(V_+^N) = \{N\}$!). As $N\geq 2$ is arbitrary, it follows that, for every $ n \geq 1$, 
\[
\lambda_{n}(H)  \le \|A\| + n.
\]
Finally, the same reasoning applies to $-H$, yielding the same bounds for $-\lambda_n(H)$, $n \leq 0$. Thus, we can conclude that
\[
|\lambda_n(H) - n| \le \|A\|, \quad n \in \mathbb{Z}.
\]
As we wanted to show.
\end{proof}


\section{Power-law ULE revisited and Proof of Theorem \ref{mainresult}}\label{subSULE}
\ 

Although it has been known by experts for a long time and discussed in recent papers, mainly by Shi et al. \cite{shi1,shi2,shi3} and Sun et al. \cite{Sun2, Sun}, that a uniform polynomial decay of eigenfunctions of operators in \( \ell^2(\mathbb{Z}) \) with pure point spectrum implies a limitation on moments, the author is not aware of any reference that provides a detailed explanation of a general notion of power-law ULE for arbitrary operators in \( \ell^2(\mathbb{Z}) \) (an analysis in this direction applied to polynomial long-range hopping random operators can be found in \cite{Sun2}). In this context, we revisit below a notion of power-law ULE for general operators with pure point spectrum acting in \( \ell^2(\mathbb{Z}) \).

\begin{lemma}[Power-law ULE]\label{PLSULEthm} Let $H$ be a self-adjoint operator acting in \( \ell^2(\mathbb{Z}) \). Suppose that there exists an orthonormal basis  of \( \ell^2(\mathbb{Z}) \) by eigenfunctions of $H$, $\{\phi_m\}_{m \in\ \mathbb{Z}}$, and an  $\alpha>0$ such that 
\begin{equation*}
|\phi_m(n)| \leq  \frac{\gamma_\alpha }{|n - m|^\alpha}, \, m,n \in \mathbb{Z},
\end{equation*}
for some \( \gamma_\alpha>0 \) depending only on $\alpha$. If $\alpha>3/2 + q/2$, then 
\[
\sup_{t \in \mathbb{R}}  \sum_{n \in \mathbb{Z}} |n|^q |\langle e^{-itH}\delta_k ,\delta_n \rangle|^2  < C_{q,k}.
\]    
\end{lemma} 

\begin{remark}{\rm  We note that it seems possible to obtain a  power-law SULE  version; to do so, one must prove a version of the Theorem 7.1 in \cite{DJLS} in the power-law case and follow the same proof of the Lemma \ref{PLSULEthm}.}
\end{remark}

\begin{proof} Let $k \in \mathbb{Z}$ and let \(\delta_k \), \( (\delta_k)(j) = \delta_{jk} \),  \( k,j \in \mathbb{Z} \). Our goal is to show that
\[
\sup_{t \in \mathbb{R}}  \sum_{n \in \mathbb{Z}} |n|^q |\langle e^{-itH}\delta_k ,\delta_n \rangle|^2 < C_{q,k}.
\]

Since the moment is given by a sum over $n \in \mathbb{Z}$, we may disregard finitely many indices.  The terms with $|n-k| \le 1$ and $|n| \le 1$ form a finite set and therefore contribute a uniformly bounded amount. Thus, we shall restrict our analysis to indices satisfying $|n-k| \ge 2$ and $|n| \ge 2$.

Note that:
\begin{eqnarray*} \langle e^{-itH}\delta_k ,\delta_n \rangle &=& \sum_{m \in \mathbb{Z}} \overline{\langle \phi_m,\delta_k \rangle} \langle e^{-itH} \phi_m  ,\delta_n \rangle\\ &=& \sum_{m \in \mathbb{Z}} e^{it\lambda_m} \overline{\langle \phi_m,\delta_k \rangle} \langle \phi_m  ,\delta_n \rangle, \, n \in \mathbb{Z}.   
\end{eqnarray*}
Since $|\phi_m(n)| \leq 1, \, m,n \in \mathbb{Z}$ and
\begin{equation*}
|\phi_m(n)| \leq  \frac{\gamma_\alpha }{|n - m|^\alpha}, \, m,n \in \mathbb{Z},
\end{equation*}
one has
\begin{eqnarray}\label{eqsule1}\nonumber |\langle e^{-itH}\psi ,\delta_n \rangle| &\leq&  \sum_{m \in \mathbb{Z}} |\langle \phi_m,\delta_k \rangle| |\langle \phi_m  ,\delta_n \rangle| = \sum_{m \in \mathbb{Z}} | \phi_m(k) | |\phi_m(n)|\\ \nonumber &=& \biggr\{\displaystyle\sum_{\substack{|m| \geq 2, \, |n -m| \geq 2 \\  |k - m|\geq  2}} | \phi_m(k) | |\phi_m(n)|\biggr\} \\ \nonumber &+&  | \phi_n(k) | |\phi_n(n)| + | \phi_k(k) | |\phi_k(n)|\\ \nonumber &+& | \phi_{n\mp 1}(k) | |\phi_{n\mp 1}(n)| + | \phi_{k\mp 1}(k) | |\phi_{k\mp 1}(n)| \\ \nonumber &+& | \phi_{-1}(k) | |\phi_{-1}(n)| + | \phi_{0}(k) | |\phi_{0}(n)| + | \phi_{1}(k) | |\phi_{1}(n)|  \\ \nonumber &\leq& \biggr\{\displaystyle\sum_{\substack{|m| \geq 2, \, |n -m| \geq 2 \\  |k - m|\geq  2}}  | \phi_m(k) | |\phi_m(n)|\biggr\}\\ &+& \nonumber \frac{2\gamma_\alpha}{|n-k|^\alpha}+ \frac{\gamma_\alpha}{|(n\mp1)-k|^\alpha} + \frac{\gamma_\alpha}{|n-(k\mp1)|^\alpha} \\   &+& \frac{\gamma_\alpha}{|n+1|^\alpha} + \frac{\gamma_\alpha}{|n|^\alpha} + \frac{\gamma_\alpha}{|n-1|^\alpha}.
\end{eqnarray}
One has also
\begin{eqnarray*}
\displaystyle\sum_{\substack{|m| \geq 2, \, |n -m| \geq 2 \\  |k - m|\geq  2}}   | \phi_m(k) | |\phi_m(n)| &\leq&   \gamma_\alpha^2 \displaystyle\sum_{\substack{|m| \geq 2, \, |n -m| \geq 2 \\  |k - m|\geq  2}}   \frac{1}{|k - m|^\alpha} \frac{1}{|n - m|^\alpha}.
\end{eqnarray*}

We analyze two distinct cases:

\noindent {\bf Case} $|k| \geq 2$: for $|n - m|\geq 2$ and $|k - m| \geq 2$, since $a,b \geq 2$ implies $a+b \leq ab$, 
\begin{eqnarray*}
|n-k| &\leq& |k - m|+|n - m|  \leq |k - m||n - m|, 
\end{eqnarray*}
for $|k - m| \geq 2$ and $|k| \geq 2$,
\[|m| = |m-k+k|\leq |m-k|+|k| \leq |m-k||k|, \]
\[\frac{|m|}{|k|}  \leq |k-m| \leq |k-m|  |n - m|.\]
Thus, in this case, for every $\epsilon>0,$
\begin{eqnarray*}
\frac{1}{|k - m|^\alpha} \frac{1}{|n - m|^\alpha} &=& \frac{1}{(|n - m||k - m|)^{(\alpha-1-\epsilon/2)+1+\epsilon/2}}\\  &=& \frac{1}{(|n - m||k - m|)^{\alpha-1-\epsilon/2}} \frac{1}{(|n - m||k - m|)^{1+\epsilon/2}}\\  &\leq& \frac{1}{|n-k|^{\alpha-1-\epsilon/2}}  \frac{|k|^{1+\epsilon/2}}{|m|^{1+\epsilon/2}} 
\end{eqnarray*}
and, therefore, 
\[ \displaystyle\sum_{\substack{|m| \geq 2, \, |n -m| \geq 2 \\  |k - m|\geq  2}}  | \phi_m(k) | |\phi_m(n)|  \leq  \frac{A_0 \gamma_\alpha^2 |k|^{1+\epsilon/2}}{|n-k|^{\alpha-1-\epsilon/2}} \]
where $A_0 = \displaystyle\sum_{m \neq 0} \frac{1}{|m|^{1+\epsilon/2}} < + \infty$.
Since $\alpha> \epsilon_0 + 3/2 + q/2 $ for some $\epsilon_0>0$, one has
\begin{eqnarray}\label{eqsule2}\nonumber
\displaystyle\sum_{\substack{|m| \geq 2, \, |n -m| \geq 2 \\  |k - m|\geq  2}} | \phi_m(k) | |\phi_m(n)|   &\leq&  \frac{A_0 \gamma_\alpha^2 |k|^{1+\epsilon_0/2}}{|n-k|^{\alpha-1-\epsilon_0/2}}\\  &\leq&     \frac{A_0 \gamma_\alpha^2 |k|^{1+\epsilon_0/2}}{|n-k|^{\frac{1+q+\epsilon_0}{2}}}.  
\end{eqnarray}
Combining \eqref{eqsule1} with \eqref{eqsule2}, is obtained
\[
\sup_{t \in \mathbb{R}}  \sum_{n \in \mathbb{Z}} |n|^q |\langle e^{-itH}\delta_k ,\delta_n \rangle|^2 < C_{q,k}.
\]

\noindent {\bf Case} $|k|\leq 1$: for $|n - m|\geq 2$, $|k - m| \geq 2$ and $|m| \geq 2$, since $|n| \geq 2$,
\begin{eqnarray*}
|n-k| &\leq& |n|+|k|\leq  |n|+2 \leq 2|n|  \leq   2(|m|+|n - m|)\\ &\leq& 2|m||n - m| \leq 2(|m-k|+|k|)|n - m|\\ &\leq&  2(|m-k|+|m-k|)|n - m| = 4 |m-k||n - m|,
\end{eqnarray*}
\[|m|-1  \leq |m| -|k| \leq |m-k| \leq |m-k||n - m|.\]
Thus, in this case, for every $\epsilon>0,$
\begin{eqnarray*}
\frac{1}{|k - m|^\alpha} \frac{1}{|n - m|^\alpha} &=& \frac{1}{(|n - m||k - m|)^{(\alpha-1-\epsilon/2)+1+\epsilon/2}}\\  &=& \frac{1}{(|n - m||k - m|)^{\alpha-1-\epsilon/2}} \frac{1}{(|n - m||k - m|)^{1+\epsilon/2}}\\  &\leq& \frac{4^{\alpha-1-\epsilon/2}}{|n-k|^{\alpha-1-\epsilon/2}}  \frac{1}{(|m|-1)^{1+\epsilon/2}} 
\end{eqnarray*}
and, therefore,

\begin{eqnarray*}
\displaystyle\sum_{\substack{|m| \geq 2, \, |n -m| \geq 2 \\  |k - m|\geq  2}}  | \phi_m(k) | |\phi_m(n)|  \leq  \frac{\Tilde{A_0} \gamma_\alpha^2 4^{\alpha-1-\epsilon/2}}{|n-k|^{\alpha-1-\epsilon/2}},   
\end{eqnarray*}
where $\Tilde{A_0} = \displaystyle\sum_{ |m| \geq 2} \frac{1}{(|m|-1)^{1+\epsilon/2}} < + \infty$. Since $\alpha> \epsilon_0 + 3/2 + q/2 $ for some $\epsilon_0>0$, one has

\begin{eqnarray}\label{eqsule3}\nonumber
\displaystyle\sum_{\substack{|m| \geq 2, \, |n -m| \geq 2 \\  |k - m|\geq  2}} | \phi_m(k) | |\phi_m(n)| &\leq& \frac{A_0 \gamma_\alpha^2 4^{\alpha-1-\epsilon_0/2}}{|n-k|^{\alpha-1-\epsilon_0/2}}\\  &\leq&  \frac{A_0 \gamma_\alpha^2 4^{\alpha-1-\epsilon_0/2}}{|n-k|^{\frac{1+q+\epsilon_0}{2}}}. 
\end{eqnarray}
Finally, combining \eqref{eqsule1} with \eqref{eqsule3}, again is obtained 

\[
\sup_{t \in \mathbb{R}}  \sum_{n \in \mathbb{Z}} |n|^q |\langle e^{-itH}\delta_k ,\delta_n \rangle|^2 < C_q.
\]
This concludes the proof of the theorem. 
\end{proof}

As mentioned in the Introduction, we conclude this section by proving Theorem \ref{mainresult}.

\subsection{Proof of Theorem \ref{mainresult}}

Note that \( i) \) follows directly from Theorem \ref{mainthm} by taking \( A = T_a +b \). Since, by Lemma \ref{PLSULEthm}, \( ii) \Rightarrow iii) \), it remains to prove \( ii) \). Let $\{\phi_m\}$ be an orthonormal basis of $\ell^2(\mathbb{Z})$ formed by eigenfunctions of ${\mathcal{L}}_0+b$, that is, $({\mathcal{L}}_0+b)\phi_m=\lambda_m \phi_m, \, m \in \mathbb{Z}$. 

Let $r \in \mathbb{N} \cup \{0\}$. By \(i)\), there exists a $\gamma>0$ such that $|\lambda_m - m - b(n)| \leq \gamma$ for all $m, n \in \mathbb{Z}$. Note that for $|m - n| \leq 2\gamma$, one has 
\[
|\phi_m(n)| \leq 1 \leq \frac{(2\gamma)^{r+1}}{|m - n|^{r+1}}.
\]
Thus, we need to estimate only the case where $|m - n| > 2\gamma$.

We proceed by induction on $r \in \mathbb{N}\cup\{0\}$. Observe that 
\[
(({\mathcal{L}}_0 + b)\phi_m)(n) = \lambda_m \phi_m(n) \iff \phi_m(n) = \frac{(T_a\phi_m) (n)}{\lambda_m-(n+b(n))}, 
\]
which implies that
\[
\phi_m(n) = \frac{\sum_{k \in \mathbb{Z}} a(k) \phi_m(n-k)}{(\lambda_m -m - b(n))-(n-m)}.
\]
Thus, we obtain 
\[
|\phi_m(n)| \leq \frac{\sum_{k \in \mathbb{Z}} |a(k)|  |\phi_m(n-k)|}{||\lambda_m -m-b(n)|- |n-m||}.
\]

Since 
\[
|\lambda_m -m-b(n)| \leq \gamma, \quad m,n \in \mathbb{Z},
\] 
for $|m - n| > 2\gamma$, we get
\begin{eqnarray*}
|\phi_m(n)| &\leq& \frac{\sum_{k \in \mathbb{Z}} |a(k)|  |\phi_m(n-k)|}{||\lambda_m -m-b(n)|- |n-m||}\\ 
&=&  \frac{\sum_{k \in \mathbb{Z}} |a(k)|  |\phi_m(n-k)|}{|n-m| - |\lambda_m -m-b(n)|} \leq \frac{\sum_{k \in \mathbb{Z}} |a(k)|  |\phi_m(n-k)|}{|n-m| - \gamma}.   
\end{eqnarray*}
Hence,
\begin{equation}\label{maineq}
|\phi_m(n)| \leq  \frac{2 \sum_{k \in \mathbb{Z}}  |a(k)||\phi_m(n-k)|}{|m-n|}, \quad |m - n| > 2\gamma.    
\end{equation}
Thus, if \( r = 0 \), so $a \in \ell_0^1(\mathbb{Z})$ and, hence, there exists \( \gamma_0' > 0 \) such that
\begin{equation*}
|\phi_m(n)| \leq  \frac{\gamma_0'}{|m-n|}, \quad |m - n| > 2\gamma.    
\end{equation*}
As previously discussed, this is sufficient to ensure that there exists \( \gamma_0 > 0 \) such that
\begin{equation*}
|\phi_m(n)| \leq  \frac{\gamma_0}{|m-n|}, \quad m,n \in \mathbb{Z}.    
\end{equation*}
This proves the case \( r = 0 \). 

Now we assume that the result holds for $r \ge 1$, that is, $a \in \ell_r^1(\mathbb{Z})$ implies
\[
|\phi_m(n)|
\le \gamma_r\,
\frac{1}
{|m - n|^{r+1}}, \, r \geq 1.
\]
We will prove that the same estimate holds for $r+1$ if $a \in \ell_{r+1}^1(\mathbb{Z})$.

Let $a \in \ell_{r+1}^1(\mathbb{Z}) \subset \ell_r^1(\mathbb{Z})$. Let $k \in \mathbb{Z}$ and suppose that $|m-n+k| \ge 2$ and $|k| \ge 2$.
In this case, we have
\begin{equation}\label{eq010102}
|m-n|
\le |(m-n)+k| + |k|
\le |(m-n)+k|\,|k|
\iff
\frac{1}{|(m-n)+k|}
\le \frac{|k|}{|m-n|}.
\end{equation}

By the induction hypothesis and \eqref{eq010102}, for
$|(m-n)+k| \ge 2$ and $|k| \ge 2$, we obtain
\begin{eqnarray*}
|\phi_m(n-k)|
&\le& \gamma_r\,
\frac{1}
{|m-n+k |^{r+1}} \\
&\le&  \gamma_r\,
\frac{1}
{|m-n|^{r+1}}\,|k|^{r+1}.
\end{eqnarray*}

Since $a(0)=0$, using \eqref{maineq}, we obtain
\begin{eqnarray}\label{maineq07}
\nonumber |\phi_m(n)|
&\le&
\frac{2\gamma_r}
{|m - n|^{r+2}}
\sum_{\substack{|(m-n)+k|\ge 2\\ |k|\ge 2}}
|a(k)|\,|k|^{r+1} \\
\nonumber&+&
\frac{2 }
{|m - n|}
|a(\mp 1)|\,|\phi_m(n \pm 1)| \\
\nonumber &+&
\frac{2}
{| m - n|}
|a(n-m)|\,|\phi_m(n-(n-m))| \\
&+&
\frac{2 }
{| m - n|}
|a((n-m)\pm 1)|\,|\phi_m(n-(n-m\pm 1))|.
\end{eqnarray}

Since $a \in \ell_{r+1}^1(\mathbb{Z})$, there exists $\gamma_r' > 0$ such that
for all $k \in \mathbb{Z}$,
\begin{equation}\label{eqfn}
|a(k)| \le \gamma_r' \frac{1}{| k|^{r+1}}.    
\end{equation}

Finally, using the induction hypothesis on the second term of the sum \eqref{maineq07} 
and applying \eqref{eqfn} to the third and fourth terms, we conclude that there exists 
a constant
$\widetilde{\gamma_r} > 0$, depending only on $ r$ and $\|a\|_{r+1}$, such that
\[
|\phi_m(n)|
\le \widetilde{\gamma_r}
\frac{1}
{|m - n|^{r+2}}.
\]
This completes the proof of the theorem.

\hfill \qedsymbol




\begin{appendices}

\section{Discrete spectrum}\label{secappendix}
\

As mentioned in the introduction, in this appendix, we prove that self-adjoint operators of the form \( H = A + V \) in \(\ell^2(\mathbb{Z})\), where \( A \) is bounded and \( V \) is such that $\displaystyle\lim_{m \to \mp\infty} V(m) = \mp\infty$, has purely discrete spectrum and its eigenvalues satisfy $\displaystyle\lim_{m \to \mp\infty} \lambda_m = \mp\infty$ (Proposition \ref{discretproposition}).

We will use the following result.

\begin{theorem}[Theorem 11.3.13. in \cite{Oliveira}]\label{thmdiscret}
Let \( H \) be a self-adjoint operator on \( \ell^2(\mathbb{Z}) \). The following statements are equivalent:
\begin{enumerate}
    \item[{\rm (a)}] There exists an orthonormal basis \( (\phi_m) \) of \( \ell^2(\mathbb{Z}) \) consisting of eigenfunctions of \( H \), that is, \( H\phi_m = \lambda_m \phi_m \), where the eigenvalues \( \lambda_m \) are real and form a discrete set. Furthermore, each one of these eigenvalues has finite multiplicity, and \( \displaystyle\lim_{m \to \mp \infty} |\lambda_m| = +\infty \).
    
    \item[{\rm (b)}] \( \left(H - zI\right)^{-1}\) is a compact operator for some \( z \in \rho(T) \) (and thus for all \( z \in \rho(T)) \).
\end{enumerate}
\end{theorem}

\subsection{Proof of Proposition \ref{discretproposition}}

The result is a consequence of the second resolvent identity. Initially, we closely followed the proof  of Theorem 4.1 in \cite{Pigossi}. Namely, let $\{\delta_m\}$ the canonical basis of $\ell^2(\mathbb{Z})$ and note that, for every $m  \in \mathbb{Z}$ and every $n \in \mathbb{Z}$, 
\[[V(\delta_m)](n)= V(n)(\delta_m)(n) = V(n) (\delta_{nm}) = V(m)\delta_{nm} =  (V(m)\delta_m)(n),\]
so \( V \) satisfies the hypotheses of the Theorem \ref{thmdiscret} and, therefore, $ \left( V - iI\right)^{-1}$  is a compact operator.    

Let us apply the second resolvent identity with:
\begin{itemize}
    \item $T = H = A+V$,
    \item $S = V$,
    \item $z = i$.
\end{itemize}

Recall the second resolvent identity states:

\[
    \left(T - iI\right)^{-1} - \left(S - iI\right)^{-1} = \left(T - iI\right)^{-1}(S - T)\left(S - iI\right)^{-1}.
\]

Substituting $T = H = A+V$, $S = V$ and $z=i$, one has 

\begin{eqnarray*} \label{eq1}
    \left(H - iI\right)^{-1} - \left(V - iI\right)^{-1} = \left(H - iI\right)^{-1} \left(-A\right) \left(V - iI\right)^{-1}.
\end{eqnarray*} 
Thus,
\begin{eqnarray*}
\label{eq2}
    \left(H - iI\right)^{-1} = \left(V - iI\right)^{-1} + \left(H - iI\right)^{-1} \left(-A\right) \left(V - iI\right)^{-1}.
\end{eqnarray*} 
Since $H$ is self-adjoint, $\left(H - iI\right)^{-1}$ is bounded since that in this case $\sigma(H) \subset \mathbb{R}$ \cite{Oliveira}. Thus, the sum above is a sum of compact operators and is, therefore, compact. Hence, by Theorem \ref{thmdiscret}, $H$ has purely  discrete spectrum, each one of the eigenvalues has finite multiplicity,  and 
\[\displaystyle\lim_{m \to \mp\infty} |\lambda_m| = +\infty.\]

Finally, let $(\phi_m)$ be an orthonormal basis consisting of eigenfunctions of $H.$ Suppose, by contradiction, that there exists $a>0$ such that
$\lambda_m \ge -a$ or $\lambda_m \le a$ for all $m \in \mathbb{Z}$.
Since, for all $n \in \mathbb Z$,
\[
(A\delta_n)(n) + V(n) = \langle H\delta_n, \delta_n \rangle
= \sum_{m\in\mathbb Z} \lambda_m |\langle \delta_n, \phi_m\rangle|^2,
\]
it follows that, for all $n \in \mathbb{Z}$,
\[
(A\delta_n)(n) + V(n) \ge -a \quad \text{or} \quad (A\delta_n)(n) + V(n) \le a.
\]
This contradicts the fact that $\displaystyle\lim_{n\to\pm\infty} V(n)=\pm\infty$, since $A$ is bounded.
Hence, since the set $\{\lambda_m\}$ is discrete,
\[
\lim_{m\to\pm\infty} \lambda_m = \pm\infty.
\]
\hfill \qedsymbol
\end{appendices}



\noindent {\bf \Large{Acknowledgments}} 

\noindent M. Aloisio was supported by grant \#2025/25338-1 from the São Paulo Research Foundation (FAPESP) and in part by grant \#01/24/APQ-03132-24 from the Minas Gerais Research Foundation (FAPEMIG). M. A. thanks César R. de Oliveira, Deberton Moura and Mariane Pigossi for reading this paper and for their valuable and constructive suggestions. Part of this work benefited from discussions of the author with Tiago Pereira and Edgar Matias at the Institute of Mathematics and Computer Sciences (ICMC), University of São Paulo. The author sincerely thanks them and the Institute for their hospitality.  The author also thanks the anonymous referee for valuable suggestions that have improved the exposition of the manuscript.

\ 

\noindent {\bf Declaration.}

\noindent {\bf Conflict of interest.}  The author has no conflict of interest/conflicting interests to declare that are relevant to the content of this article.

\end{document}